\documentclass[reqno, 10pt]{amsart}
\usepackage{amssymb}
\usepackage{amsmath}
\usepackage[mathscr]{euscript}
\usepackage{hyperref}
\usepackage{graphicx}
\usepackage[normalem]{ulem}

\usepackage{amssymb}
\usepackage{hyperref}
\RequirePackage[dvipsnames]{xcolor} 
\definecolor{halfgray}{gray}{0.55} 
\definecolor{webgreen}{rgb}{0,0.5,0}
\definecolor{webbrown}{rgb}{.6,0,0} \hypersetup{%
  colorlinks=true, linktocpage=true, pdfstartpage=3,
  pdfstartview=FitV,%
  breaklinks=true, pdfpagemode=UseNone, pageanchor=true,
  pdfpagemode=UseOutlines,%
  plainpages=false, bookmarksnumbered, bookmarksopen=true,
  bookmarksopenlevel=1,%
  hypertexnames=true,
  pdfhighlight=/O,
  urlcolor=webbrown, linkcolor=RoyalBlue,
  citecolor=webgreen, 
  pdftitle={},%
  pdfauthor={},%
  pdfsubject={2000 MAthematical Subject Classification: Primary:},%
  pdfkeywords={},%
  pdfcreator={pdfLaTeX},%
  pdfproducer={LaTeX with hyperref}%
}

\usepackage{enumitem}

\theoremstyle{plain}
\newtheorem{theorem}{Theorem}[section]
\newtheorem{lemma}[theorem]{Lemma}
\newtheorem{corollary}[theorem]{Corollary}

\theoremstyle{definition}

\newtheorem{remark}[theorem]{Remark}

\DeclareMathOperator{\Ima}{Im}

\DeclareMathOperator{\Id}{Id}

\def\R{\mathbb{R}}
\def\N{\mathbb{N}}

\def\B{\mathcal{B}}
\newcommand{\norm}[1]{{\left\lVert \, #1 \, \right\rVert}}

\begin{document}

\title
[On the spectral radius of compact operator cocycles]
{On the spectral radius of compact operator cocycles}

\author{Lucas Backes}

\address{Departamento de Matem\'atica, Universidade Federal do Rio Grande do Sul, Av. Bento Gon\c{c}alves 9500, CEP 91509-900, Porto Alegre, RS, Brazil.}
\email{lucas.backes@ufrgs.br}

\author{Davor Dragi\v cevi\'c}
\address{Department of Mathematics, University of Rijeka, 51000 Rijeka, Croatia}
\email{ddragicevic@math.uniri.hr}

\date{\today}


\begin{abstract}
We extend the notions of joint and generalized spectral radii to cocycles acting on Banach spaces and obtain a version of Berger-Wang's formula when restricted to the space of cocycles taking values in the space of compact operators. Moreover, we observe that the previous quantities depends continuously on the underlying cocycle.
\end{abstract}

\maketitle

\section{Introduction}
Let $M_d$ be the space of all real $d\times d$ matrices and consider a compact set $\mathcal M\subset M_d$. Then, we can define  the \emph{joint spectral radius}  of $\mathcal M$ by
\[
\hat{\rho}(\mathcal M)=\lim_{n\to \infty} \sup \{\lVert A_n \cdots A_1 \rVert^{1/n}:  A_i\in \mathcal M\}.
\]
This notion was introduced by Rota and Strang in their seminal paper~\cite{RS} and since then it has found its applications in several different fields like coding theory~\cite{MOS} and stability theory~\cite{Dai12}.
 Furthermore, one can define a  \emph{generalized spectral radius} of $\mathcal M$  by
\[
\bar{\rho}(\mathcal M)=\limsup_{n\to \infty} \sup\{ \rho (A_n\cdots A_1)^{1/n}: A_i \in \mathcal M \},
\]
 where $\rho(A)$ denotes the usual spectral radius of $A\in M_d$.  

The celebrated result of Berger and Wang~\cite{BW} (usually called the  Berger-Wang formula) asserts that those two quantities  coincide, i.e. $\hat{\rho}(\mathcal M)=\bar{\rho}(\mathcal M)$. Moreover, this equality holds even when  $\mathcal M$ is just a  bounded subset of $M_d$.
In addition, several results concerned with the regularity of the map $\mathcal M\mapsto \hat{\rho}(\mathcal M)$ (acting on the space of compact subsets of $M_d$) were obtained. Indeed, Wirth~\cite{W} proved that this map is continuous and also established its local Lipschitz continuity on the space of irreducible compact sets $\mathcal M \subset M_d$
(explicit Lipschitz constant was given in~\cite{K}). 

It was  natural to ask  whether these results can be extended to the infinite-dimensional setting, where $\mathcal M$ is  a compact subset of the space of all bounded operators acting on some Banach space $\B$.   It turns out that in this setting, the version of the Berger-Wang formula doesn't hold in general. Indeed, Gurvits~\cite{Gu95} provided  an explicit counterexample
(with $\mathcal M$ consisting of only two operators).  However, some partial extensions of the Berger-Wang formula were obtained in~\cite{ST00, ST02} with the best result to this date being that of Morris~\cite{IM}.

It turns out that the previously mentioned results can be formulated in the context of ergodic theory. Indeed, it is possible to associate to $\mathcal M$, the so-called \emph{linear cocycle} (of matrices or operators) which acts over a full two-sided shift $(M, f)$
and to give an alternative formulation of the Berger-Wang formula and related results. We refer to Remark~\ref{918} for a detailed discussion.  This observation also opened possibilities of
using tools from ergodic theory to study the notions of joint and generalized spectral radius from the point of view of dynamical systems. In this direction,  Dai~\cite{Dai14} obtained the version of the  Berger-Wang formula for linear cocycles of matrices acting over subshifts $(M, f)$ of finite-type. More recently, Zou, Cao and Liao~\cite{ZCL18} extended the result of Dai
by proving that the same conclusion holds whenever $(M, f)$ is a topological dynamical system satisfying the so-called Anosov Closing property. Moreover, they showed that the joint spectral radius is a continuous function on the space of H\"{o}lder continuous cocycles. 

The main objective of the present paper is to extend the results in~\cite{ZCL18} to the case of linear cocycles with values in the space of compact operators acting on arbitrary Banach spaces. This is achived by carefully combining various results in the literature. In particular, our recent 
results dealing with the approximation of Lyapunov exponents in the infinite-dimensional setting~\cite{BD} play a central role.

\section{Preliminaries}\label{sec: statements}

Throughout this paper $(M,d)$  will be a compact metric space and $f: M \to M$ will be a homeomorphism that satisfies the \textit{Anosov Closing property}. We recall that the latter  means that there exist $C_1 ,\varepsilon _0 ,\theta >0$ such that if $z\in M$ satisfies $d(f^n(z),z)<\varepsilon _0$ then there exists a periodic point $p\in M$ such that $f^n(p)=p$ and
\begin{displaymath}
d(f^j(z),f^j(p))\leq C_1 e^{-\theta \min\lbrace j, n-j\rbrace}d(f^n(z),z),
\end{displaymath}
for every $j=0,1,\ldots ,n$. We note that shifts of finite type, basic pieces of Axiom A diffeomorphisms and more generally, hyperbolic homeomorphisms are particular examples of maps satisfying the Anosov Closing property. We refer to~\cite[Corollary 6.4.17.]{KH95} for details.

\subsection{Semi-invertible operator cocycles}

Let $(\B,\norm{\cdot})$ be a Banach space and let $B(\B,\B)$ be the space of all bounded linear maps from $\B$ to itself. Denote by $B_0(\B,\B)$ the subset of $B(\B,\B)$ formed by the compact operators of $\B$. We recall that $B(\B,\B)$  is  a Banach space with respect to the norm
\[
 \norm{T} =\sup \lbrace \norm{Tv}/\norm{v};\; \norm{v}\neq 0 \rbrace, \quad T\in B(\B,\B)
\]
and $B_0(\B,\B)$ is a closed subspace of $(B(\B,\B),\|\cdot\|)$. Although we use the same notation for the norms on $\B$ and $B(\B,\B)$ this will not cause any confusion. Finally,  consider a map  $A:M\to B(\B,\B)$.

The \emph{semi-invertible operator cocycle} (or just \textit{cocycle} for short) generated by $A$ over $f$ is  defined as the map $A:\mathbb{N}\times M\to B(\B,\B)$ given by
\begin{equation*}\label{def:cocycles}
A^n(x):=A(n, x)=
\left\{
	\begin{array}{ll}
		A(f^{n-1}(x))\ldots A(f(x))A(x)  & \mbox{if } n>0 \\
		\Id & \mbox{if } n=0,\\
	\end{array}
\right.
\end{equation*}
for all $x\in M$. 
The term `semi-invertible' refers to the fact that the action of the underlying dynamical system $f$ is assumed to be an invertible transformation while the action on the fibers given by $A$ may fail to be invertible. 

We say that the cocycle generated by $A$ over $f$ is \emph{compact} if $A$ take values in $B_0(\B,\B)$, i.e. if $A(x)\in B_0(\B, \B)$ for each $x\in M$.

\subsection{Volume growth} Let $T\in B(\B,\B)$. For a subspace $V\subset \B$, set
$$m(T_{\mid V}):=\inf \{ \|Tv\|; v\in V \text{ with } \|v\|=1 \}.$$
Furthermore, for each $k\in \N$ such that $k\le d:=\dim \mathcal B$, set
\[
F_k(T):=\sup \{ m(T_{\mid V}); \; V\subset \B \text{ is a $k$-dimensional subspace}\},
\]
\[
c_k(T):=\inf \{\|T_{\mid V}\|; \; V\subset \B \text{ is a $(k-1)$-codimensional subspace} \},
\]
and
\[
V_k(T)=\sup \left\lbrace \text{det}(T\mid_V); V\subset \B \text{ is a $k$-dimensional subspace} \right\rbrace,
\]
where 
\begin{equation*}
\text{det}(T\mid_V)=
\left\{
	\begin{array}{ll}
		\frac{m_{TV}(T(B_V))}{m_V(B_V)}  & \mbox{if } T\mid_V \text{ is injective, } \\
		0 & \mbox{otherwise},\\
	\end{array}
\right.
\end{equation*}
and $m_V$ denotes the \emph{Haar measure} on the subspace $V$ normalized so that the unit ball $B_V$ in $V$ has measure given by the volume of the Euclidean unit ball in $\R^k$.  
We recall that quantities $F_k(T)$ are called \emph{Kolmogorov numbers} of $T$, while  quantities $c_k(T)$ are called \emph{Gelfand numbers} of $T$.

We note  that $V_k(T)$, $\prod_{j=1}^k F_j(T)$ and $\prod_{j=1}^k c_j(T)$ may be interpreted as  the growth rates of $k$-dimensional volumes spanned by $\{Tv_i\}_{i=1}^k$, where $v_i\in \B$ are unit vectors. Below we present a result relating all the previous notions of volume growth. In fact, this result says that  up to a  multiplicative constant, all of them  coincide.

\begin{lemma}\label{lemma: relation volume growth}
Given $k\in \N$ such that $k\le \dim \mathcal B$, there exists $C>0$ (depending only on $k$) such that for every $T\in B(\B,\B)$,
$$\frac{1}{C}F_k(T)\leq c_k(T)\leq C F_k(T),$$
$$\frac{1}{C}V_k(T)\leq \prod_{j=1}^k F_j(T)\leq C V_k(T)$$
and 
$$\frac{1}{C}V_k(T)\leq \prod_{j=1}^k c_j(T)\leq C V_k(T).$$
\end{lemma}
\begin{proof}
The first estimate is proved in~\cite[Lemma 15.]{BM}, while the second  is established in the proof of~\cite[Lemma A.2]{DFGTV}.  Finally, the last assertion of the lemma is an easy consequence of the first two. 
\end{proof}

We shall also need the following auxiliary  result. 
\begin{lemma}\label{lemma: subadditive}
For every $k\in \N$, the map $T\to V_k(T)$ is continuous  on $B(\B, \B)$. Moreover, it is  also 
submultiplicative, i.e. 
\[
V_k(TS)\le V_k(T)V_k(S) \quad \text{for every $T, S\in B(\B, \B)$.}
\]
\end{lemma}
\begin{proof}
The continuity of the map $T\to V_k(T)$ is established in~\cite[Lemma 2.20]{AB}, while the submultiplicativity property was observed in the proof of~\cite[Lemma A.2]{DFGTV}.
\end{proof}

\subsection{Multiplicative ergodic theorem}
We now recall the version of the multiplicative ergodic theorem established in~\cite{FLQ13} (see also~\cite{AB, GTQ}).  We stress that we don't state it in full generality. Indeed, we present a simplified version  that will be sufficient for our purposes. 
\begin{theorem}
Assume that $A$ is a continuous and compact cocycle over $f$. Furthermore, let $\mu$ be an $f$-invariant ergodic probability measure on $M$. Then, there exists a Borel set $\mathcal R^\mu \subset M$ such that $\mu(\mathcal R^\mu)=1$ and either:
\begin{enumerate}
\item there is a finite sequence of numbers 
\[
\lambda_1(A,\mu) >\lambda_2(A,\mu)>\cdots >\lambda_k(A,\mu)>\lambda_\infty(A,\mu)=-\infty
\]
and a measurable decomposition
\[
\mathcal B=E_1(x) \oplus \cdots \oplus E_k(x) \oplus E_\infty(x)
\]
such that for $x\in \mathcal R^\mu$,
\[
A(x) E_i(x) = E_i(f(x)), \quad i=1,\ldots,k,  
\]
\[
A(x) E_\infty(x) \subset E_\infty(f(x)), 
\]
and
\[
 \lim_{n \to \infty} \frac 1 n \log \lVert A^n(x)v\rVert=\lambda_i(A,\mu),
\]
for $v\in E_i(x)\setminus \{0\}$, $i\in \{1, \ldots k, \infty\}$. Moreover, each $E_i(x), i=1,\ldots,k$, is a finite-dimensional subspace of $\mathcal B$;
\item there exists an infinite sequence of numbers
\[
\lambda_1 (A,\mu)>\lambda_2(A,\mu) >\cdots >\lambda_k(A,\mu) > \ldots  >\lambda_\infty(A,\mu)=-\infty,
\]
\[
\lim_{k\to \infty}\lambda_k(A,\mu)=\lambda_\infty (A,\mu),
\]
and for each $k\in \mathbb N$ a  measurable decomposition
\[
\mathcal B=E_1(x) \oplus \cdots \oplus E_k(x) \oplus \ldots \oplus E_\infty(x)
\]
such that for $x\in \mathcal R^\mu$,
\[
A(x) E_i(x)= E_i(f(x)), \quad  i\in \N, 
\]
\[
A(x)E_\infty(x)\subset E_\infty(f(x)), 
\]
and
\[
\lim_{ n\to \infty} \frac 1 n\log \lVert A^n(x)v\rVert=\lambda_i(A,\mu), 
\]
for $v\in E_i(x)\setminus \{0\}$, $i\in \mathbb N\cup \{\infty\}$.
Moreover, each $E_i(x), i\in \N$ is a finite-dimensional subspace of $\mathcal B$.
\end{enumerate}

\end{theorem}
Recall that numbers $\lambda_i(A,\mu)$ are called \emph{Lyapunov exponents} of $A$ with respect to $\mu$. Moreover, $d_i(A,\mu):=\dim E_i(x)$  is said to be a \emph{multiplicity} of $\lambda_i(A,\mu)$.

We observe that the previous result holds under a much more general assumption on the base dynamics, namely, $f$ only has to be invertible and measurable. In particular, the Anosov Closing property is not necessary for it. Nevertheless, this property plays an important part in our results because it allow us to give a simplified description of the Lyapunov spectrum in terms of the behaviour of the cocycle on periodic orbits (see for instance \cite{Kal11,Bac18,BD}).

\subsection{The Joint and Generalized spectral radii} 
Finally, we recall some notions that  will be of central importance to our paper. 
Given $s>0$ and $T\in B(\B,\B)$, let $d=\text{dim}(\B)$ and consider
$$\varphi^s_c(T)=\begin{cases}
c_1(T)c_2(T)\cdots c_{\lfloor s\rfloor}(T)c_{\lfloor s\rfloor+1}(T)^{s-\lfloor s\rfloor} & \text{for $s< d$;}\\
\lvert \det T\rvert^{s/d} & \text{for $s\ge d$ and $d<\infty$;}
\end{cases}$$
and
$$\varphi^s_F(T)= \begin{cases}F_1(T)F_2(T)\cdots F_{\lfloor s\rfloor}(T)F_{\lfloor s\rfloor+1}(T)^{s-\lfloor s\rfloor} & \text{for $s<d$;}\\
\lvert \det T\rvert^{s/d} & \text{for $s\ge d$ and $d<\infty$.}
\end{cases}$$
\begin{remark}\label{rmk1}
Assume that $d<\infty$. Then, $c_i(T)$ is precisely the  $i$-th singular value of $T$ for $i=1, \ldots, d$ (see~\cite{BM} for example).  Hence, $\varphi^s_c(T)$ coincides with $\varphi^s(T)$, where $\varphi^s$ is a singular value function (see~\cite[p.2]{ZCL18}). The singular value function was introduced by Falconer \cite{F88} in the study of Hausdorff dimension of self-affine fractals and since then it have been playing an important part in the study of dimension theory \cite{F94}. For instance, in \cite{BM18} the authors describe the equilibrium states associated to this function for certain families of matrices which in turn are connected with the study of Hausdorff dimension of some self-affine sets (see also the references therein). Bearing in mind the importance of this function in the finite dimensional setting, it is natural to generalize it to infinite dimensions and investigate its properties and applications.
\end{remark}

Then, we define the \emph{s-joint spectral radius of $(A,f)$} by 
\begin{equation*}
\begin{split}
\hat{\rho}_s(A)&:= \lim _{n \to +\infty} \sup_{x\in M} \varphi^s_c(A^n(x))^{\frac{1}{n}}\\
&=\lim _{n \to +\infty} \sup_{x\in M} \varphi^s_F(A^n(x))^{\frac{1}{n}},
\end{split}
\end{equation*}
where the second equality follows from Lemma~\ref{lemma: relation volume growth}. Finally, the \emph{s-generalized spectral radius of $(A,f)$} is  defined by
$$\overline{\rho}_s(A):=\limsup_{n\to +\infty} \left( \sup_{x\in M} \rho_s(A^n(x))\right)^{\frac{1}{n}},$$ 
where \[\rho_s(T):= \lim_{n\to +\infty} \varphi _c^s(T^n)^{\frac{1}{n}}\quad  \text{for any $T\in B(\B,\B)$. }\]

\begin{remark}
It follows from Remark~\ref{rmk1}  that in the case when $d<\infty$, $\hat{\rho}_s(A)$ and $\overline{\rho}_s(A)$ coincide respectively  with the values of the joint spectral radius and the  generalized spectral radius of $A$ which were studied  in~\cite{ZCL18}. For finite dimensional linear operators $\hat{\rho}_1(A)$ and $\overline{\rho}_1(A)$ have seen important applications, notably in the fundamental work of Daubechies and Lagarias \cite{DL92i, DL92ii} on wavelets regularity. In the infinite dimensional setting, these objects have also seen many usefull applications, for instance, to establish a number of results in invariant subspace theory \cite{S84, ST00}. 
\end{remark}

\section{Main results}
We are now in position to formulate the main results of our paper.  For $\alpha >0$, 
we say that $A:M\to B(\B,\B)$ is  an  \emph{$\alpha$-H\"{o}lder continuous map} if there  exists a constant $C_2>0$ such that
\begin{displaymath}
\norm{A(x)-A(y)} \leq C_2 d(x,y)^{\alpha},
\end{displaymath}
for all $x,y\in M$.

The following is our first result. It can be described as an extension of~\cite[Theorem C.]{ZCL18} to the case of compact cocycles acting on arbitrary Banach spaces. 
\begin{theorem}\label{theo: berger-wang} 

Let $f: M\to M $ be a homeomorphism satisfying the Anosov Closing property and $A:M\to B_0(\B,\B)$ an $\alpha$-H\"{o}lder continuous map. Then,
$$\hat{\rho}_s(A)=\overline{\rho}_s(A)$$
for every $s>0$.
\end{theorem}

\begin{remark}\label{918}
The first result in the spirit of Theorem~\ref{theo: berger-wang} was obtained by Berger and Wang~\cite{BW} and, in fact, their result is a particular case of Theorem~\ref{theo: berger-wang}. More precisely, let $\mathcal M$ be a compact  subset of $M_d$, where $M_d$ denotes the space of all real matrices of order $d$. Set $M=\mathcal M^{\mathbb Z}$
and equip $M$ with the product topology so that it becomes a compact metric space. Furthermore, let $f\colon M\to M$ be a two-sided shift given  by $f((M_i)_{i\in \mathbb Z})=(M_{i+1})_{i\in \mathbb Z}$ for $(M_i)_{i\in \mathbb Z}\in M$.  Finally, let  $A\colon M\to M_d$ be given by $A((M_i)_{i\in \mathbb Z})=M_0$, $(M_i)_{i\in \mathbb Z}\in M$. 
It turns out that the main result from~\cite{BW} can be recovered from Theorem~\ref{theo: berger-wang}  for this particular choice of $M$, $f$ and $A$ and $s=1$. We note that strictly speaking, the main result from~\cite{BW} requires only that $\mathcal M \subset M_d$ is bounded. However, it turns out that this general version can be deduced from the one previously stated 
by replacing $\mathcal M$ with its closure (see~\cite[p.8]{IM}). Similar results for the case when $A$ is as above but when $(M,f)$ is a subshift of finite type were obtained by Dai~\cite{Dai14}. Finally, as we already mentioned,  the general case that corresponds to our Theorem~\ref{theo: berger-wang}  when $\mathcal B$ is finite-dimensional was treated in~\cite{ZCL18}.

It is also worth noticing that, as pointed out by Gurvits~\cite[Theorem A.1.]{Gu95}, the previously described result by Berger and Wang (and consequently our Theorem \ref{theo: berger-wang}) doesn't hold, in general, in the infinite-dimensional case. In fact, Gurvits presented an example of two operators $T_1,T_2\in B(\B,\B)$ for which the generalized spectral radius is strictly smaller than the joint spectral radius. However, some partial extensions of Berger-Wang's formula to the infinite dimensional setting were obtained in~\cite{ST00, ST02}, with the most general result being that of Morris~\cite[Theorem 1.4.]{IM}. In the particular case when dealing with compact operators, the result of Morris is covered by Theorem~\ref{theo: berger-wang} and it corresponds to the case when $M=\mathcal M^{\mathbb Z}$, where $\mathcal M$ is a (pre)compact subset of $B_0(\mathcal B, \mathcal B)$ and with $f$ and $A$ as in the previous paragraph and $s=1$.  In Theorem~\ref{theo: berger-wang}, we deal with a general case when $(M, f)$ is any topological dynamical system satisfying Anosov Closing property and where $A$ is an arbitrary $\alpha$-H\"{o}lder continuous compact cocycle. 
\end{remark}

For $\alpha >0$, set
\[
C^\alpha (M, B_0(\B, \B))  :=\bigg{\{} A \colon M\to B_0(\B, \B): \text{$A$ is an  $\alpha$-H\"{o}lder continuous map} \bigg{\}}.
\]
 We note that $C^\alpha (M, B_0(\B, \B))$ is a Banach space with respect to the norm
\[
\lVert A\rVert_\alpha:=\sup_{x, y\in M}\lVert A(x)-A(y)\rVert+\sup_{x\neq y} \bigg{(} \frac{\lVert A(x)-A(y)\rVert}{d(x, y)^\alpha}\bigg{)}.
\]
The following is our second result. It represent an extension of~\cite[Theorem A.]{ZCL18}   to the case of compact cocycles acting on arbitrary Banach spaces. 
\begin{theorem}\label{theo: continuity} 
Let $f: M\to M $ be a homeomorphism satisfying the Anosov Closing property. Then, the map
$$A\to \hat{\rho}_s(A)=\overline{\rho}_s(A)$$
is continuous on $C^\alpha (M, B_0(\B, \B)) $.
\end{theorem}

\begin{remark}
We stress that the conclusion of Theorem~\ref{theo: continuity}  can fail if $f$ doesn't satisfy the Anosov Closing property. Indeed, explicit counterexamples were constructed in~\cite{DHH17, WY13} (see~\cite[p.2]{ZCL18} for a detailed discussion).

\end{remark}

\begin{remark}
Finally, we would like to explain why we restricted our attention to the case of compact cocycles. It turns out (see~\cite[Theorem 2.1.]{Deg08}) that the spectral radius mapping is not a continuous function on $B(\B, \B)$.  Hence, Theorem~\ref{theo: continuity}  doesn't hold if one  replaces
$C^\alpha (M, B_0(\B, \B))$ by $C^\alpha (M, B(\B, \B))$ even in the case when $A$ is a constant map. 
\end{remark}

\section{Proofs}
In this section we present proofs of our main results. 
\subsection{Proof of Theorem \ref{theo: continuity}}
We can assume that $\B$ is infinite-dimensional since the case when $d<\infty$ is covered by~\cite[Theorem A.]{ZCL18}.
 We first present several auxiliary results. 
\begin{lemma} \label{lemma: log of joint spectral raidus}
For any $s\in \N$,
\begin{displaymath}
\begin{split}
\log \hat{\rho}_s(A)&=\max_{\mu \in \mathcal{M}_f} \left\lbrace \inf_n \frac{1}{n}\int \log V_s(A^n(x))d\mu \right\rbrace  \\
&=\max_{\mu \in \mathcal{M}_f} \left\lbrace \lim_n \frac{1}{n}\int \log \varphi^s_c(A^n(x))d\mu   \right\rbrace \\
&=\max_{\mu \in \mathcal{M}_f} \left\lbrace \lim_n \frac{1}{n}\int \log \varphi^s_F(A^n(x))d\mu  \right\rbrace,
\end{split}
\end{displaymath}
where $\mathcal{M}_f$ denotes the set of all $f$-invariant probability measures.
\end{lemma}
\begin{proof}
The last two equalities follow directly from Lemma~\ref{lemma: relation volume growth}. Let us now proof that the first  equality holds. 
By Lemma~\ref{lemma: relation volume growth}, we have  that 
\begin{displaymath}
\begin{split}
\log \hat{\rho}_s(A)&=\log \left( \lim_{n\to +\infty} \sup_{x\in M}  V_s(A^n(x))^{\frac{1}{n}} \right)  \\
&=\lim_{n\to +\infty} \sup_{x\in M} \left( \frac{1}{n}\log V_s(A^n(x)) \right).
\end{split}
\end{displaymath}
It follows from Lemma~\ref{lemma: subadditive}  that we can apply~\cite[Lemma A.3]{IM13} for $f_n(x)=\log V_s(A^n(x))$ and we obtain that 
\begin{displaymath}
\begin{split}
\lim_{n\to +\infty} \sup_{x\in M} \left( \frac{1}{n}\log V_s(A^n(x)) \right)&=\inf_{n} \sup_{x\in M} \left( \frac{1}{n}\log V_s(A^n(x)) \right)\\
&=\max_{\mu \in \mathcal{M}_f}\left\lbrace \inf_{n}  \frac{1}{n} \int \log V_s(A^n(x)) d\mu \right\rbrace.
\end{split}
\end{displaymath}
The proof of the lemma is completed. 
\end{proof}

\begin{lemma} \label{lemma: Lyap exp}
For any $s\in \N$,
\begin{displaymath}
\inf_n \frac{1}{n}\int \log V_s(A^n(x))d\mu =\gamma_1(A,\mu)+\gamma_2(A,\mu)+\cdots +\gamma_s(A,\mu),
\end{displaymath}
where $\gamma_j(A,\mu)$ stands for the $j$-th Lyapunov exponent of $(A,f)$ with respect to $\mu$ counted with multiplicities. In particular,
$$\log \hat{\rho}_s(A)=\max_{\mu \in \mathcal{M}_f} \left\lbrace \gamma_1(A,\mu)+\gamma_2(A,\mu)+\cdots +\gamma_s(A,\mu)\right\rbrace. $$
\end{lemma}
\begin{proof}
The first claim was established in the proof of~\cite[Lemma A.3]{DFGTV},  while the second is a direct consequence of the first one  together with Lemma~\ref{lemma: log of joint spectral raidus}.
\end{proof}

\begin{lemma} \label{lemma: continuity mutlidim spectral radius}
For any $s\in \N$ and $T\in B(\B,\B)$, set
\begin{displaymath}
r_s(T):=\lim_{n\to +\infty} \frac{1}{n} \log V_s(T^n).
\end{displaymath}
Then, $T\to r_s(T)$ is a continuous map on $B_0(\B, \B)$.
\end{lemma}

\begin{proof} 
We first observe that it follows from Lemma~\ref{lemma: subadditive} that $r_s(T)$ is well-defined for each $T\in B(\B, \B)$.
By Lemma~\ref{lemma: relation volume growth}, we have that 
\begin{displaymath}
r_s(T)=\lim_{n\to +\infty} \frac{1}{n} \log V_s(T^n) = \lim_{n\to +\infty} \frac{1}{n} \sum_{j=1}^s \log c_j(T^n).
\end{displaymath}
In particular, $r_1(T)$ is just the logarithm of the spectral radius of $T$. Thus, it follows from \cite{CM79} (see also \cite[Theorem 2.1]{Deg08})  that $T\to r_1(T)$ is  a continuous map on $B_0(\B, \B)$.

In order to treat the general case,  we start by recalling some classical facts about compact operators on Banach spaces. Let $T$ be  a compact operator acting on $\B$. Since we assumed that $\B$ is infinite-dimensional, 
 it follows from the Fredholm's Alternative (see~\cite[Theorem 6.2.8]{RR00}) that its spectrum $\sigma(T)$ can be written as 
$$\sigma(T)=\{0\}\cup \{\lambda_i: i\in \mathbb N\}$$
with
\[
|\lambda_1|\geq|\lambda_2|\geq |\lambda_3|\geq \ldots \quad \text{and} \quad \lim_{n\to \infty} \lambda_n=0
\]
where  each $\lambda_i$ is an   eigenvalue of $T$. Moreover, from the Riesz Decomposition Theorem for compact operators (see \cite[Theorem 6.4.11]{RR00} and~\cite[Corollary 6.4.12]{RR00}),  we conclude  that for every $i\in \mathbb N$, 
$$\B=\mathcal{N}_{\lambda_i}\bigoplus \mathcal{R}_{\lambda_i},$$
where $\mathcal{N}_{\lambda_i}=\text{ker}(T-\lambda_i)^N$ and $\mathcal{R}_{\lambda_i}=\Ima (T-\lambda_i)^N$ for some $N\in \mathbb{N}$. Furthermore, $\mathcal{N}_{\lambda_i}$ and $\mathcal{R}_{\lambda_i}$ are invariant under $T$, $\mathcal{N}_{\lambda_i}$ is finite-dimensional, $\sigma(T_{\mid \mathcal{N}_{\lambda_i}})=\{\lambda_i\}$ and $\sigma(T_{\mid \mathcal{R}_{\lambda_i}})=\sigma(T)\setminus \{\lambda_i\}$. Set 
\begin{displaymath}
\begin{split}
\xi_j(T)&:= \lim_{n\to +\infty} \frac{1}{n}  \log c_j(T^n) \\
&= \lim_{n\to +\infty} \frac{1}{n}  \log V_j(T^n) -\lim_{n\to +\infty} \frac{1}{n}  \log V_{j-1}(T^n),
\end{split}
\end{displaymath}
for $j\in \mathbb N$.  Observe that
\begin{equation}\label{828}
r_s(T)=\sum_{j=1}^s \xi_j(T).
\end{equation}
We claim that
\begin{equation}\label{803}
\xi_j(T) =\log \lvert \lambda_k\rvert, 
\end{equation}
for $k\in \mathbb N$ and $j\in \{ \text{dim} (N_{\lambda_{0}})+\ldots+\text{dim} (N_{\lambda_{k-1}})+1, \ldots, \text{dim} (N_{\lambda_{0}})+\ldots+\text{dim} (N_{\lambda_{k}}) \}$, where $\text{dim}(N_{\lambda_0}):=0$.  Indeed, the fact that~\eqref{803} holds for $j=1$ was already observed. Take now $j=2$. If $\text{dim}(\mathcal{N}_{\lambda_1})=1$,  then one can conclude that $\xi_2(T)=\log \lvert \lambda_2\rvert$ by applying~\eqref{803} for $j=1$ and  $T_{\mid \mathcal{R}_{\lambda_1}}$ (instead of $T$).
Assume now that  $\text{dim}(\mathcal{N}_{\lambda_1})\geq 2$. Then, if $V\subset \B$ is a $1$-codimensional subspace we have that $V\cap \mathcal{N}_{\lambda_1}\neq \{0\}$.  Let us fix $v_V\in \mathcal{N}_{\lambda_1}\cap V$ such that $\lVert v_V\rVert=1$. We have that 
\[
c_2(T^n)=\inf_V\lVert (T^n)_{\rvert_V}\rVert \ge \inf_V \lVert T^n v_V\rVert \ge \lVert ((T_{\mid \mathcal{N}_{\lambda_1}})^{-1})^n\rVert^{-1}.
\]
Hence,
\[
\begin{split}
\xi_2(T) &=\lim_{n\to +\infty} \frac{1}{n}  \log c_2(T^n) \ge -\lim_{n\to +\infty} \frac{1}{n}  \log \lVert ((T_{\mid \mathcal{N}_{\lambda_1}})^{-1})^n\rVert =\log |\lambda_1|,
\end{split}
\]
since $r_1((T_{\mid \mathcal{N}_{\lambda_1}})^{-1})=\lvert \lambda_1\rvert^{-1}$. 
This easily implies that that $\xi_2(T)=\log |\lambda_1|$. Hence, \eqref{803} holds for $j=2$. By iterating the above argument one can conclude that~\eqref{803} holds  for each $j\in \mathbb N$.

Finally, since $T\to \sigma(T)$ is continuous on $B_0(\B, \B)$ (see \cite{CM79}),  the conclusion of the lemma follows from~\eqref{828} and~\eqref{803}.
\end{proof}

We are now in position to complete the proof of~Theorem \ref{theo: continuity}.
Suppose initially that $s\in \N$. 
By Lemma~\ref{lemma: subadditive} we know that $(A,\mu)\to \frac{1}{n} \int \log V_s(A^n(x))d\mu $ is  a continuous map for every $n\in \N$.
 In particular, $A\to \inf_n \frac{1}{n} \int \log V_s(A^n(x))d\mu $ is upper-semicontinuous. Thus, the compactness of $\mathcal{M}_f$ combined with Lemma~\ref{lemma: log of joint spectral raidus} implies that $A\to \hat{\rho}_s(A)$ is upper-semicontinuous. 

On the other hand, by Lemma~\ref{lemma: Lyap exp} combined with~\cite[Theorem 2.5]{BD} we obtain that
\begin{equation} \label{eq: joint spectral radius X Lyap on periodic points}
\begin{split}
\log \hat{\rho}_s(A)&=\max_{\mu \in \mathcal{M}_f} \left\lbrace \gamma_1(A,\mu)+\gamma_2(A,\mu)+\cdots +\gamma_s(A,\mu)\right\rbrace \\
&= \max_{\mu \in \mathcal{M}_f(Per)} \left\lbrace \gamma_1(A,\mu)+\gamma_2(A,\mu)+\cdots +\gamma_s(A,\mu)\right\rbrace,
\end{split}
\end{equation}
where $\mathcal{M}_f(Per)$ denotes  the set of all $f$-invariant  probability measures supported on periodic orbits. Now, if $p\in M$ satisfies $f^k(p)=p$ and $\mu_p$ is the $f$-invariant measure supported on the orbit of $p$,  then 
\begin{equation}\label{eq: lyap X r_s}
\begin{split}
\gamma_1(A,\mu_p)+\gamma_2(A,\mu_p)+\cdots +\gamma_s(A,\mu_p)&= \frac{1}{k} r_s(A^k(p)). 
\end{split}
\end{equation}
Indeed, it follows from Lemmas~\ref{lemma: subadditive} and~\ref{lemma: Lyap exp} together with Kingman's subadditive ergodic theorem that
\begin{equation*}
\begin{split}
\gamma_1(A,\mu_p)+\gamma_2(A,\mu_p)+\cdots +\gamma_s(A,\mu_p)&= \inf_n \frac{1}{n}\int \log V_s(A^n(x))d\mu_p \\
&= \lim_{n\to +\infty} \frac{1}{n} \log V_s(A^n(p)) \\
&= \lim_{n\to +\infty} \frac{1}{nk} \log V_s(A^{nk}(p)) \\
&= \frac{1}{k}\lim_{n\to +\infty} \frac{1}{n} \log V_s(A^{k}(p)^n) \\
&= \frac{1}{k} r_s(A^k(p)). 
\end{split}
\end{equation*}
We observe that Lemma~\ref{lemma: continuity mutlidim spectral radius} implies that the map  $A\to \frac{1}{k} r_s(A^k(p))$ is continuous and consequently, 
 the map  $A\to \log \hat{\rho}_s(A)$ is lower-semicontinuous which yields the conclusion of the theorem in the case when $s\in \mathbb N$.

Take now an arbitrary $s>0$.   Observe that
\begin{equation}\label{1119}
\varphi_c^s(T)=(\varphi_c^{\lfloor s\rfloor+1}(T))^{s-\lfloor s\rfloor}(\varphi_c^{\lfloor s\rfloor}(T))^{1-s+\lfloor s\rfloor},
\end{equation}
for any $T\in B(\B, \B)$.  By setting $V_s(T):=(V_{\lfloor s\rfloor+1}(T))^{s-\lfloor s\rfloor} (V_{\lfloor s\rfloor}(T))^{1-s+\lfloor s\rfloor}$, one can repeat the arguments in the proof of Lemma~\ref{lemma: log of joint spectral raidus} to show that
\[
\log \hat{\rho}_s(A)=\max_{\mu \in \mathcal{M}_f} \left\lbrace \inf_n \frac{1}{n}\int \log V_s(A^n(x))d\mu \right\rbrace  =
\max_{\mu \in \mathcal{M}_f} \left\lbrace \lim_n \frac{1}{n}\int \log \varphi^s_c(A^n(x))d\mu   \right\rbrace.
\] 
Arguing as in the case when $s\in \N$, we obtain that $A\mapsto \log \hat{\rho}_s(A)$ is upper-semicontinuous. 

On the other hand, it follows from~\eqref{1119} that
\begin{align*}
\log \hat{\rho}_s(A) &=\max_{\mu \in \mathcal{M}_f} \left\lbrace (s-\lfloor s\rfloor) \lim_n \frac{1}{n}\int \log \varphi^{\lfloor s\rfloor +1}_c(A^n(x))d\mu \right. \\
& \left. +(1-s+\lfloor s\rfloor)  \lim_n \frac{1}{n}\int \log \varphi^{\lfloor s\rfloor }_c(A^n(x))d\mu \right\rbrace \\
&\ge \max_{f^k(p)=p, k\in \N}\left\lbrace (s-\lfloor s\rfloor) \lim_n \frac{1}{n}\int \log \varphi^{\lfloor s\rfloor +1}_c(A^n(x))d\mu_p \right. \\
&  \left.   +(1-s+\lfloor s\rfloor)  \lim_n \frac{1}{n}\int \log \varphi^{\lfloor s\rfloor }_c(A^n(x))d\mu_p \right\rbrace.
\end{align*}

Thus,
\[
\log \hat{\rho}_s(A) \ge \max_{f^k(p)=p, k\in \N}\left\lbrace  \frac{s-\lfloor s\rfloor }{k}  r_{\lfloor s\rfloor+1}(A^k(p))+\frac{1-s+\lfloor s\rfloor }{k} r_{\lfloor s\rfloor}(A^k(p))\right\rbrace.
\]
By applying~\cite[Theorem 2.5]{BD} one can easily conclude that
\[
\log \hat{\rho}_s(A) \le \max_{f^k(p)=p, k\in \N}\left\lbrace  \frac{s-\lfloor s\rfloor }{k}  r_{\lfloor s\rfloor+1}(A^k(p))+\frac{1-s+\lfloor s\rfloor }{k} r_{\lfloor s\rfloor}(A^k(p))\right\rbrace,
\]
and therefore
\begin{equation}\label{1214}
\log \hat{\rho}_s(A) = \max_{f^k(p)=p, k\in \N}\left\lbrace  \frac{s-\lfloor s\rfloor }{k}  r_{\lfloor s\rfloor+1}(A^k(p))+\frac{1-s+\lfloor s\rfloor }{k} r_{\lfloor s\rfloor}(A^k(p))\right\rbrace.
\end{equation}
Arguing as in the case when $s\in \N$, we obtain that  the map  $A \mapsto \hat{\rho}_s(A)$ is  lower-semicontinuous. Hence, the conclusion of the theorem holds for arbitrary $s>0$.

\begin{remark}\label{816}
We note that it was not necessary to refer to~\cite{ZCL18} for the case when $d<\infty$.  Indeed, our arguments can be easily modified to cover the  finite-dimensional case also.  In fact, one only needs to modify slightly (actually simplify) the proof of Lemma~\ref{lemma: continuity mutlidim spectral radius}. Moreover, observe that the hypothesis that $f$ satisfies the Anosov Closing property was only used to apply the results from \cite{BD}.
\end{remark}

\subsection{Proof of Theorem \ref{theo: berger-wang} }
Let us again assume that $\B$ is infinite-dimensional (Remark~\ref{816} applies for the proof of this theorem also).
We start with two auxiliary lemmas.

\begin{lemma}
For any  $T\in B_0(\B,\B)$ and $j, n\in \mathbb N$, we have 
$$\xi_j(T^n)=n\xi_j(T).$$
\end{lemma}
\begin{proof} 
Using the same notation as in the proof of Lemma~\ref{lemma: continuity mutlidim spectral radius} one has that 
\[
\sigma (T^n)=\{\lambda_i^n: i\in \N\}.
\]
Furthermore, $\text{dim} N_{\lambda_i}=\text{dim} N_{\lambda_i^n}$. Hence, the conclusion of the lemma follows directly from~\eqref{803} (applied both for $T$ and $T^n$).
\end{proof}

\begin{lemma} \label{lemma: equality varphi and rho}
For any $s\in \N$ and  $T\in B_0(\B,\B)$,
$$\lim_{n\to +\infty}  \varphi_c^s(T^n)^{\frac{1}{n}} = \limsup_{n\to +\infty}\rho_s(T^n)^{\frac{1}{n}}. $$
\end{lemma}

\begin{proof} 
Observe that 
\begin{displaymath}
\lim_{n\to +\infty} \frac{1}{n} \log \varphi_c^s(T^n)=\lim_{n\to +\infty} \frac{1}{n}\log V_s(T^n)=r_s(T).
\end{displaymath}
On the other hand,
\begin{displaymath}
\begin{split}
\frac{1}{n}\log \rho _s(T^n)&=\frac{1}{n} \log \lim_{m\to +\infty}V_s((T^n)^m)^{\frac{1}{m}}\\
&=\frac{1}{n} \lim_{m\to +\infty} \frac{1}{m}\log V_s((T^n)^m)\\
&= \frac{1}{n}r_s(T^n).
\end{split}
\end{displaymath}
Thus, since
$$r_s(T^n)=\sum_{j=1}^s\xi_j(T^n),$$
it follows from the previous lemma that for each $n\in \N$, 
$$r_s(T^n)=n\sum_{j=1}^s\xi_j(T)=nr_s(T).$$
This completes the proof of the lemma. 
\end{proof}
 We start observing that from Lemma \ref{lemma: relation volume growth} and the submultiplicativity of $V_s$ (see Lemma \ref{lemma: subadditive}) we get that for any $T\in B(\B,\B)$ and $s\in \N$,
\begin{displaymath}
\begin{split}
\rho_s(T)&=\lim_{n\to +\infty} \left( c_1(T^n)c_2(T^n)\ldots c_s(T^n)\right)^{\frac{1}{n}} \\
&=\lim_{n\to +\infty} V_s(T^n)^{\frac{1}{n}}=\inf_n V_s(T^n)^{\frac{1}{n}} \\ 
& \leq V_s(T).
\end{split}
\end{displaymath}
Take now any $s>0$.  It follows easily from~\eqref{1119} that
\[
\rho_s(T)=\rho_{\lfloor s\rfloor+1}(T)^{s-\lfloor s\rfloor} \rho_{\lfloor s\rfloor}(T)^{1-s+\lfloor s\rfloor}, \quad \text{for  every $T\in B(\B, \B)$.}
\]
Using~Lemma~\ref{lemma: relation volume growth} and the previous observation, we have that
\[
\begin{split}
\overline{\rho}_s(A) &=\limsup_{n\to +\infty} \left( \sup_{x\in M} \rho_s(A^n(x))\right)^{\frac{1}{n}} \\
&=\limsup_{n\to +\infty} \left( \sup_{x\in M} \rho_{\lfloor s\rfloor+1}(A^n(x))^{s-\lfloor s\rfloor} \rho_{\lfloor s\rfloor}(A^n(x))^{1-s+\lfloor s\rfloor}\right)^{\frac{1}{n}} \\
&\le \limsup_{n\to +\infty} \left( \sup_{x\in M} V_{\lfloor s\rfloor +1}(A^n(x))^{s-\lfloor s\rfloor} V_{\lfloor s\rfloor}(A^n(x))^{1-s+\lfloor s\rfloor}\right)^{\frac{1}{n}} \\
&=  \limsup_{n\to +\infty} \left( \sup_{x\in M} \varphi_c^{\lfloor s\rfloor+1}(A^n(x))^{s-\lfloor s\rfloor} \varphi_c^{\lfloor s\rfloor}(A^n(x))^{1-s+\lfloor s\rfloor}\right)^{\frac{1}{n}} \\
&=\limsup_{n\to +\infty} \left( \sup_{x\in M} \varphi_c^{s}(A^n(x))\right)^{\frac{1}{n}} \\
&=\limsup_{n\to +\infty} \sup_{x\in M}\varphi_c^{s}(A^n(x))^{\frac 1 n},
\end{split}
\]
and therefore
\begin{equation}\label{211}
\overline{\rho}_s(A)  \le \hat{\rho}_s(A).
\end{equation}

Let us now establish the converse inequality.  Take an arbitrary $p\in \text{Fix}(f^k)$, where $k\in \N$. By  Lemma~\ref{lemma: equality varphi and rho}, we have that
\[
\begin{split}
\log \overline{\rho}_s(A)&=\limsup_{n\to +\infty} \sup_{x\in M} \log  \rho_s(A^n(x))^{\frac{1}{n}} \\
&\ge \limsup_{n\to +\infty}  \log  \rho_s(A^n(p))^{\frac{1}{n}} \\
&\ge \limsup_{n\to +\infty}  \log  \rho_s(A^k(p)^n)^{\frac{1}{kn}} \\
&=\frac{s-\lfloor s\rfloor}{k} \limsup_{n\to +\infty}  \log  \rho_{\lfloor s\rfloor+1} (A^k(p)^n)^{\frac{1}{n}} \\
&\phantom{=}+\frac{1-s+\lfloor s\rfloor}{k} \limsup_{n\to +\infty}  \log  \rho_{\lfloor s\rfloor} (A^k(p)^n)^{\frac{1}{n}} \\
&=\frac{s-\lfloor s\rfloor}{k} \limsup_{n\to +\infty}  \log  \varphi_c^{\lfloor s\rfloor+1} (A^k(p)^n)^{\frac{1}{n}} \\
&\phantom{=}+\frac{1-s+\lfloor s\rfloor}{k} \limsup_{n\to +\infty}  \log  \varphi_c^{\lfloor s\rfloor} (A^k(p)^n)^{\frac{1}{n}} \\
&=\frac{s-\lfloor s\rfloor}{k} r_{\lfloor s\rfloor+1}(A^k(p))+\frac{1-s+\lfloor s\rfloor}{k}r_{\lfloor s\rfloor}(A^k(p)).
\end{split}
\]
Hence, \eqref{1214} implies that $\log \overline{\rho}_s(A) \ge \log \hat{\rho}_s(A)$. Therefore, 
\[
\overline{\rho}_s(A) \ge \hat{\rho}_s(A),
\]
which together with~\eqref{211} yields the conclusion of the theorem. \qed

As a consequence of our proofs, we also get the following ``more dynamical" result.

\begin{corollary}
Let $f$ and $A$ be as before. Then,
\begin{displaymath}
\hat{\rho}_s(A)=\overline{\rho}_s(A)= \limsup_n \sup_{k\geq 1} \sup_{p\in \text{Fix}(f^k)} \rho_s (A(f^{nk-1}(p))\cdots A(p))^{\frac{1}{nk}}.
\end{displaymath}
\end{corollary}

\medskip{\bf Acknowledgements.} 
We thank to the referee for useful comments on the first version of this work. L.B. was partially supported by a CNPq-Brazil PQ fellowship under Grant No. 306484/2018-8. D.D. was supported in part by Croatian Science Foundation under the project IP-2019-04-1239 and by the University of Rijeka under the projects uniri-prirod-18-9 and uniri-prprirod-19-16.

\end{document}